\documentclass[12pt, reqno]{amsart}

\usepackage{hyperref}
\usepackage{amssymb,upref,enumerate}
\usepackage{tikz}
\usepackage{mathabx}
\usepackage{amsmath,amssymb,yfonts,amscd,amsthm,amsxtra,mathrsfs}
\usepackage{latexsym}
\usepackage{soul}
\usepackage{verbatim}
\usepackage[T1]{fontenc}
\usepackage{tgtermes}

\usepackage{cite} 

\makeatletter
\def\Ddots{\mathinner{\mkern1mu\raise\p@
\vbox{\kern7\p@\hbox{.}}\mkern2mu
\raise4\p@\hbox{.}\mkern2mu\raise7\p@\hbox{.}\mkern1mu}}
\makeatother

\def\Xint#1{\mathchoice
{\XXint\displaystyle\textstyle{#1}}%
{\XXint\textstyle\scriptstyle{#1}}%
{\XXint\scriptstyle\scriptscriptstyle{#1}}%
{\XXint\scriptscriptstyle\scriptscriptstyle{#1}}%
\!\int}
\def\XXint#1#2#3{{\setbox0=\hbox{$#1{#2#3}{\int}$}
\vcenter{\hbox{$#2#3$}}\kern-.5\wd0}}

\def\dashint{\Xint-}

\newtheorem{theorem}{Theorem}[section]

\newtheorem{corollary}[theorem]{Corollary}

\theoremstyle{definition}

\newtheorem{lemma}[theorem]{Lemma}

\newtheorem{remark}[theorem]{Remark}

\def\ep{{\epsilon}}

\def\al{{\alpha}}
\def\R{\mathbb R}
\def\N{\mathbb N}

\def\C{\mathbb C}
\def\Z{\mathbb Z}
\def\Sn{\mathbb S}

\def\ra{\rightarrow}
\def\vp{\varphi}
\def\bey{\begin{eqnarray*}}
\def\eey{\end{eqnarray*}}

\allowdisplaybreaks

\begin{document}

\title[]{Pointwise estimates for rough operators with applications to weighted Sobolev inequalities}
\author{Cong Hoang}

\address{Cong Hoang \\
 Department of Mathematics \\
Florida Agricultural and Mechanical University \\
 Tallahassee, FL 32307, USA}

\email{cong.hoang@famu.edu}

\author{Kabe Moen}

\address{Kabe Moen \\
 Department of Mathematics \\
 University of Alabama \\
 Tuscaloosa, AL 35487, USA}

\email{kabe.moen@ua.edu}

\author{Carlos P\'erez }

\address{Carlos P\'erez \\
 Department of Mathematics \\
 University of the Basque Country \\
 Ikerbasque and BCAM \\
 Bilbao, Spain}

\email{cperez@bcamath.org}


\maketitle

 \begin{abstract} We investigate weighted Sobolev inequalities for rough operators.  We prove that several operators satisfy a pointwise bound by the Riesz potential applied to the gradient.  From this inequality we derive several new Sobolev type inequalities with an operator on the left hand side.  

\end{abstract}

\section{Introduction} 
\subsection{Motivation}
The classical Gagliardo-Nirenberg-Sobolev inequality in $\R^n$ for $n\geq 2$ is given by
\begin{equation}\label{GNSineq} \|f\|_{L^{p^*}(\R^n)}\leq C\|\nabla f\|_{L^p(\R^n)}\end{equation}
where $1\leq p<n$, $p^*=\frac{np}{n-p}$, and $f\in C^\infty_c(\R^n)$.  Inequality \eqref{GNSineq} is essential in the study of PDE, and in fact implies the Sobolev embedding $\dot W^{1,p}(\R^n)\hookrightarrow L^{p^*}(\R^n)$. One approach to proving \eqref{GNSineq}, probably the most common, is to use the well-known pointwise estimate 
\begin{equation}\label{represent} |f(x)|\leq {\color{blue} c_n} \,I_1(|\nabla f|)(x)\end{equation} 
(see \cite{Sal, GT, KLV}) where $I_1$ is the Riesz potential operator of order $\al=1$ given by 
$$I_\al f(x)=\int_{\R^n}\frac{f(y)}{|x-y|^{n-\al}}\,dy, \quad 0<\al<n,$$
combined with the boundedness of $I_1$ from $L^p(\R^n)$ into $L^{p^*}(\R^n)$, $1<p<n$.
%
%
A relevant extension of \eqref{GNSineq} is due to Muckenhoupt and Wheeden \cite{MW} 
who showed that 
\begin{equation}\label{fracbdd}\|wI_1 f\|_{L^{p^*}(\R^n)}\leq c\,\|wf\|_{L^p(\R^n)}, \end{equation}
holds if and only if $w\in A_{p,p^*}$
$$[w]_{A_{p,p^*}}=\sup_Q \left(\dashint_Q w^{p^*} \right) \left(\dashint_Q w^{-p'} \right)^{\frac{p^{*}}{p'}}<\infty.$$
By the pointwise bound \eqref{represent} and the boundedness of $I_1$ from \eqref{fracbdd} we have the following weighted Gagliardo-Nirenberg-Sobolev inequality
\begin{equation}\label{oneweightSob}
\|wf\|_{L^{p^*}(\R^n)}\leq C\|w\nabla f\|_{L^{p}(\R^n)}, \quad w\in A_{p,p^*}
\end{equation}
in the range $1<p<n$. 

To derive another extension of \eqref{GNSineq} we consider the usual Hardy-Littlewood maximal function defined by 
$$Mf(x)=\sup_{Q\ni x} \dashint_Q |f(y)|\,dy.$$
The action of operators on Sobolev spaces began with Kinnunen \cite{Kin} who showed that $M:W^{1,p}(\R^n)\ra W^{1,p}(\R^n)$ for $1<p<\infty$. An easier mapping property, which simply follows from inequality \eqref{GNSineq} and the boundedness of $M$ on $L^{p^*}(\R^n)$ since $p^*>1$, is the fact that $M:\dot W^{1,p}(\R^n)\ra L^{p^*}(\R^n).$  Consider the weighted version,
\begin{equation}\label{maxSob}\|wMf\|_{L^{p^*}(\R^n)}\leq C\|w\nabla f\|_{L^p(\R^n)}.\end{equation}
To prove \eqref{maxSob} we could use Muckenhoupt's $A_p$ theorem to conclude that $M$ is bounded on $L^{p^{*}}(w^{p^{*}})$ since $w\in A_{p,p^*}$ implies that $w^{p^*}\in A_{p^*}$ (see Section \ref{oneweight}) and then apply inequality \eqref{oneweightSob}.  However, we can take a different approach by proving a stronger pointwise estimate using the theory of $A_1$ weights. Recall a weight $w$ belongs to the class $A_1$ if
$$Mw\leq Cw.$$  
where the smallest constant $C$ is denoted $[w]_{A_1}$. It is well-known that power weights of the form $|x|^{-a}$, for $0<a<n$, are prototypical examples of $A_1$ weights.  The operator $I_1f$, being a convolution with an $A_1$ weight, is also an $A_1$ weight when $f\geq 0$. This observation was noticed by Sawyer \cite{Saw1} but we include a proof (see also Theorem \ref{A1thm}).  Indeed, for any convolution operator we have
$$M(f*\vp)\leq |f|*(M\vp)$$
and hence if $\vp\in A_1$ and $f\geq 0$, then $[f*\vp]_{A_1}\leq [\vp]_{A_1}.$ With this fact in hand we obtain the following improvement of inequality \eqref{represent}:
\begin{equation}\label{pointmax}Mf(x)\leq CM\big(I_1(|\nabla f|)\big)(x)\leq C I_1(|\nabla f|)(x), \end{equation}
where the constant depends only on $n$.  Again, the boundedness of $I_1$ from \eqref{fracbdd} and the pointwise inequality \eqref{pointmax} clearly imply \eqref{maxSob} which also improves \eqref{oneweightSob} since $Mf\geq |f|$.   Iterating \eqref{pointmax} leads to the pointwise bound

\begin{equation}\label{maxiteratedptwise}M^kf(x)\leq C_kI_1(|\nabla f|)(x), \end{equation}
for any $k\in \N$ and the inequality
\begin{equation}\label{maxiteratedSob}\|wM^kf\|_{L^{p^*}(\R^n)}\leq C_k\|w\nabla f\|_{L^p(\R^n)}, \end{equation}
which further improves  \eqref{oneweightSob}. This method has the advantage that it yields a better dependence on $[w]_{A_{p,p^*}}$ (see Section \ref{oneweight}).  


There is yet another way to improve inequality \eqref{maxSob}, hence also inequality \eqref{oneweightSob}, which requires two additional maximal functions. The sharp maximal function is defined by
$$M^\#f(x)=\sup_{Q\ni x}\dashint_Q |f-f_Q|$$
where $f_Q=\dashint_Q f$ and the fractional maximal function
$$M_\al f(x)=\sup_{Q\ni x} \ell(Q)^\al \dashint_Q |f|, \quad 0\leq \al<n.$$
The operator $M^\#$ plays a central role in the study of Calder\'on-Zygmund operators, while $M_\al$ is the natural maximal function that is related to the Riesz potential $I_\al$. The $(1,1)$ Poincar\'e inequality 
$$\dashint_Q|f-f_Q|\leq C\ell(Q)\left(\dashint_Q |\nabla f|\right)$$
holds for all cubes $Q$ and leads to the pointwise inequality
\begin{equation}\label{sharp} M^\#f(x)\leq CM_1(|\nabla f|)(x)\end{equation}
(this observation  can also be found in \cite{Tor}). The operator $M_1$ is pointwise smaller than $I_1$ and thus shares the same boundedness properties. While the pointwise bound $M^\# f\leq 2Mf$ holds, Fefferman and Stein \cite{FS} showed that the reverse inequality holds in norm, namely, when $0<p<\infty$ and $w\in A_\infty$
\begin{equation}\label{FS} \|Mf\|_{L^p(w)}\leq C\|M^\#f\|_{L^p(w)}.\end{equation}
Combining inequalities \eqref{sharp} and \eqref{FS} yields the following new inequality
\begin{equation}\label{FSgrad} \|Mf\|_{L^p(w)}\leq C\|M_1(|\nabla f|)\|_{L^p(w)}\end{equation}
for $0<p<\infty$ and $w\in A_\infty$. Inequality \eqref{FSgrad} also implies the weighted maximal operator Sobolev inequality \eqref{maxSob}, albeit with a worse dependence on the constant $[w]_{A_{p,p^*}}$, since both $w^{p^*}\in A_\infty$ and $M_1:L^p(w^p)\ra L^{p^*}(w^{p^*})$ hold when $w\in A_{p,p^*}$.  We also remark that \eqref{FSgrad} could be derived via a Coifman-Fefferman 
type estimate between $I_1$ and $M_1$ (see Section \ref{oneweight}).

Inequality \eqref{FSgrad} also has some other interesting consequences. The Hardy-Littlewood maximal function controls several operators, including Calder\'on-Zygmund operators, via the Coifman-Fefferman inequality (see \cite{CF})
\begin{equation}\label{CoifFeff}\|Tf\|_{L^p(w)}\leq C\|Mf\|_{L^p(w)}, \quad w\in A_\infty,\ 0<p<\infty.\end{equation}
Inequalities \eqref{FS} and \eqref{CoifFeff} are related and hold for more general weights than $A_\infty$, known as $C_p$ weights, although a complete characterization is still open (see Lerner \cite{Ler}). Hence if $w\in A_\infty$, $0<p<\infty$, and $T$ is a Calder\'on-Zygmund operator, by combining \eqref{FSgrad} and \eqref{CoifFeff} we have
  \begin{equation}\label{CFgrad} \|Tf\|_{L^p(w)}\leq C\|M_1(|\nabla f|)\|_{L^p(w)}, \quad w\in A_\infty, \ \ 0<p<\infty.\end{equation}
Inequality \eqref{CFgrad} again implies the Sobolev bound for Calder\'on-Zygmund operators
\begin{equation}\label{CZOSob}
\|wTf\|_{L^{p^*}(\R^n)}\leq C\|w\nabla f\|_{L^p(\R^n)}
\end{equation}
when $w\in A_{p,p^*}$.  

This method works for many other operators $T$ satisfying \eqref{CoifFeff}:
\begin{itemize}
\item weakly strongly singular integral operators, as considered by C.~Fefferman \cite{fefferman} (see \cite{AHP,CLPR});\\
\item pseudo-differential operators---more precisely, pseudo-differential operators in the H\"ormander class \cite{hormander67} (see \cite{AHP,CLPR});\\ 
\item oscillatory integral operators of the kind introduced by Phong and Stein \cite{phong-stein85} (see \cite{AHP,CLPR});\\ 
\item square functions $g_\lambda^*$, (see \cite{CP});\\
\item rough operators $T_{\Omega}$ with $\Omega\in L^\infty(\mathbb S^{n-1})$ or even for $B_{(n-1)/2}$, the Bochner--Riesz multiplier at the critical index $(n-1)/2$ (see \cite{LPRR});\\
\end{itemize}
from which  \eqref{CZOSob} holds for any $T$ as above.

A variation of the above procedure can be done for commutators of singular integrals with BMO functions. 
Indeed, it was shown in \cite{Per97} that 
$$\|T^k_b f\|_{L^p(w)}\leq C\|M^{k+1}f\|_{L^p(w)}$$ 
where $T_b^k$  is the higher order commutator of a Calder\'on-Zygmund operator with a BMO function $b$
and  $w\in A_\infty$ and $0<p<\infty$.  The case $k=0$ corresponds to \eqref{CoifFeff}.
Using \eqref{maxiteratedSob} we have the Sobolev inequality for commutators as in \eqref{CZOSob}, namely
\begin{equation*}
\|wT^k_b f\|_{L^{p^*}(\R^n)}\leq C\|w\nabla f\|_{L^p(\R^n)}
\end{equation*}
when $w\in A_{p,p^*}$ and $p>1$.

\subsection{The main operators and the main result.} 

The purpose of this paper is to study Sobolev inequalities such as \eqref{CZOSob} with different operators on the lefthand side but with more precise bounds in terms of the $A_{p,p^{*}}$ constants. This will be  the case of several classical smooth convolution singular integrals but we will be considering operators beyond Calder\'on-Zygmund operators like the so-called rough singular integral operators.  These are operators 
of the form

\begin{equation*}\label{roughSIdef} T_\Omega f(x)=p.v.\int_{\R^n}\frac{\Omega(y)}{|y|^n}f(x-y)\,dy\end{equation*}
where $\Omega$ is a degree zero homogeneous function belonging to $L^1(\mathbb S^{n-1})$ with $\int_{\mathbb S^{n-1}}\Omega \,d\sigma=0$.  The corresponding maximal operator is given by
$$M_\Omega f(x)=\sup_{t>0}\frac{1}{t^n}\int_{|y|<t}|\Omega(y)f(x-y)|\,dy$$
where $\Omega$ is as above, but does not necessarily have mean zero on $\Sn^{n-1}$. The study of rough singular integrals was initiated by Calder\'on and Zygmund in \cite{CZ}. For a general $\Omega$, these operators fall outside the realm of Calder\'on-Zygmund theory due to the lack of kernel smoothness.  When $\Omega\in L^r(\mathbb S^{n-1})$ for $r>1$, or less restrictive assumptions, the $L^p(\R^n)$ boundedness of $T_\Omega$ follows from the method of rotations or Littlewood-Paley theory combined with Fourier analytic estimates.  The weak endpoint estimates are significantly more difficult and have been studied by several authors including Christ \cite{Cr}, Hofmann \cite{H}, Christ and Rubio de Francia \cite{CRdF}, and Seeger \cite{Se}.  Weighted estimates for rough operators are also delicate and of interest to several mathematicians including Duoandidoetxea and Rubio de Francia \cite{DRdF}, Duoandikoetxea \cite{Duo}, Vargas \cite{Var}, and Watson \cite{Wat}. In particular, Watson \cite{Wat} showed that if $\Omega\in L^r(\mathbb S^{n-1})$ has zero average, then
\begin{equation}\label{weightrough}T_\Omega:L^q(w)\ra L^q(w), \quad q>r', \ \ w\in A_{\frac{q}{r'}}.\end{equation}

Consider the inequality
\begin{equation}\label{onewTOmegaSob}\|wT_\Omega f\|_{L^{p^*}(\R^n)}\leq C\|w\nabla f\|_{L^p(\R^n)}\end{equation}
when $w\in A_{p,p^*}$. If $\Omega\in L^n(\Sn^{n-1})$ and $1<p<n$ we have $n'<p^*<\infty$ and $w\in A_{p,p^*}$ implies $w^{p^*}\in A_{1+\frac{p^*}{p'}}=A_{\frac{p^*}{n'}}$. By the boundedness \eqref{weightrough} we have $T_{\Omega}:L^{p^*}(w^{p^*})\ra L^{p^*}(w^{p^*})$.  Combining this with the embedding $\dot W^{1,p}(w^p)\hookrightarrow L^{p^*}(w^{p^*})$ from \eqref{oneweightSob} we obtain inequality \eqref{onewTOmegaSob}.  It is not clear whether we can obtain the inequality \eqref{CFgrad} for $T_\Omega$ when $\Omega\in L^n(\mathbb S^{n-1})$ because the version of Coifman-Fefferman inequality (11) has been shown recently but only when $\Omega\in L^\infty(\Sn^{n-1})$ (see \cite{CCDPO,LPRR}).

For an operator $T$, the two weight Sobolev inequality
\begin{equation}\label{twoSISob} \|T f\|_{L^q(u)}\leq C\|\nabla f\|_{L^p(v)}\end{equation}
is also very challenging to prove because the previous approaches do not readily extend to \eqref{twoSISob} without strong \emph{a-priori} assumptions on the individual weights such as $u\in A_q$ or $u\in A_\infty$. By the pointwise bound \eqref{pointmax}, the two weight Sobolev inequality for the Hardy-Littlewood maximal operator,
$$\|Mf\|_{L^q(u)}\leq C\|\nabla f\|_{L^p(v)}$$ 
will hold whenever $I_1$ is bounded from $L^p(v)$ to $L^q(u)$. This may include spaces where the maximal function is unbounded on the individual space $L^q(u)$.  Instead we show that for mild conditions on $\Omega$, the pointwise inequality \eqref{represent} can be extended to $T_\Omega$, i.e.,
\begin{equation*}|T_\Omega f(x)|\leq CI_1(|\nabla f|)(x).\end{equation*}

The pointwise bound by $I_1(|\nabla f|)$ is natural and holds for several of the classical operators in harmonic analysis. First, it holds for the Hardy-Littlewood maximal operator \eqref{pointmax} and its iterates \eqref{maxiteratedptwise}. In the complex plane, consider the Cauchy transform 
$$Cf(z)=\frac{1}{\pi}\int_\C\frac{f(\zeta)}{z-\zeta}\,dA(\zeta)$$
and the Ahlfors-Beurling operator
$$Sf(z)=-\frac{1}{\pi}p.v.\int_\C\frac{f(\zeta)}{(z-\zeta)^2}\,dA(\zeta).$$
The operators $S$ and $C$ are related via the equality
$$S=\partial \circ C=C\circ \partial$$
where $\partial$ is the complex derivative $\partial=\tfrac12(\partial_x-i\partial_y)$
(see the book by Astala, Iwaniec, and Martin \cite[Chapter 4]{AIM}).  The Riesz potential of a complex variable ($n=2$ and $\al=1$) is given
$$I_1f(z)=\int_{\C}\frac{f(\zeta)}{|z-\zeta|}\,dA(\zeta).$$
Then
\begin{multline*}|Sf(z)|=\big|C\big(\partial f\big)(z)\big|=\frac1\pi\left|\int_{\C}\frac{\partial f(\zeta)}{z-\zeta}\,dA(\zeta)\right|\\
\leq \frac1\pi\int_{\C}\frac{|\partial f(\zeta)|}{|z-\zeta|}\,dA(\zeta)=CI_1(|\partial f|)(z)\leq CI_1(|D f|)(z)\end{multline*}
where
$$|D f|^2=|{\partial_x u}|^2+|{\partial_y u}|^2+|{\partial_x  v}|^2+|{\partial_y v}|^2$$
for $u=\mathsf{Re}\, f$ and $v=\mathsf{Im}\, f$.  In particular, if $f$ is real valued then $|{\partial f}|=\frac12|\nabla f|$ and hence \begin{equation}\label{Beur} |Sf(z)|\leq CI_1(|\nabla f|)(z).\end{equation} 

In $\R^n$ consider the Riesz transforms
$$R_jf(x)=p.v.\int_{\R^n}\frac{x_j-y_j}{|x-y|^{n+1}}f(y)\,dy, \quad j=1,\ldots,n.$$
By using the well-known fact that 
$$c_nR_j=\partial_{x_j}\circ I_1=I_1\circ\partial_{x_j},$$
which can be seen via the Fourier transform, we have the pointwise inequality
\begin{equation}\label{riesz} |R_j f(x)|=C|I_1\big(\partial_{x_j}f \big)(x)|\leq CI_1(|\nabla f|)(x).\end{equation}

The pointwise bounds \eqref{Beur} and \eqref{riesz} can be extended to the maximally truncated operators
$$S^\star f(z)=\sup_{t>0}\left|\frac1\pi\int_{|\zeta|>t}\frac{f(z-\zeta)}{\zeta^2}\,dA(\zeta)\right|, \quad z\in \C$$
and
$$R_j^\star f(x)=\sup_{t>0}\left|\int_{|y|>t}\frac{y_j}{|y|^{n+1}}f(x-y)\,dy\right|, \quad j=1,\ldots,n, \ x\in \R^n.$$
Indeed, Mateu and Verdera \cite{MV} showed that the maximal Ahlfors-Beurling operator satisfies
\begin{equation}\label{maxbeur} S^\star f(z)\leq C M(Sf)(z)\end{equation}
and the maximal Riesz transforms satisfy
\begin{equation} \label{maxriesz} R_j^\star f(x)\leq CM^2(R_jf)(x).\end{equation}
The difference between the kernels of the Ahlfors-Beurling operator and the Riesz transforms is what yields the different iterates of the maximal function. Namely the kernel of the Ahlfors-Beurling operators is even whereas the kernel of Riesz transforms is odd.  These inequalities are derived for general operators in the seminal papers \cite{BMO, MOPV, MOV}.  The pointwise bounds \eqref{maxbeur} and \eqref{maxriesz} coupled with \eqref{Beur} and \eqref{riesz} and the fact that $I_1(|\nabla f|)$ is an $A_1$ weight now imply
$$S^\star f(z)\leq CM(Sf)(z)\leq CM(I_1(|\nabla f|))(z)\leq CI_1(|\nabla f|)(z)$$
and 
$$R_j^\star f(x)\leq CM^2(R_jf)(x)\leq CM^2(I_1(|\nabla f))(x)\leq CI_1(|\nabla f|)(x).$$

The pointwise bound by $I_1(|\nabla f|)$ also holds for more exotic maximal operators. Consider the spherical maximal operator, introduced by Stein \cite{Ste}
$$\mathcal M_{\Sn^{n-1}}f(x)=\sup_{t>0}\left|\int_{\Sn^{n-1}}f(x-ty')\,d\sigma(y')\right|$$
where $\sigma$ is the usual measure on the unit sphere.  If $f\in C_c^\infty(\R^n)$ then  
\begin{multline*} \int_{\Sn^{n-1}}f(x-ty')\,d\sigma(y')=-\int_t^\infty \frac{d}{dr}\int_{\Sn^{n-1}}f(x-ry')\,d\sigma(y')dr \\
=\int_t^\infty\int_{\Sn^{n-1}}\nabla f(x-ry')\cdot y' d\sigma(y')dr=\int_{|y|>t} \frac{\nabla f(x-y)\cdot y}{|y|^n}\,dy.
\end{multline*}
(see \cite{HL} for similar calculations), and we have
$$\mathcal M_{\Sn^{n-1}}f(x)\leq I_1(|\nabla f|)(x).$$
Haj\st{l}asz and Liu \cite{HL2} extended this to more general maximal operators.  They show that if $\mu$ is a compactly supported measure in a ball about the origin that satisfies
\begin{equation}\label{growth} \mu(B(x,r))\leq Cr^{n-1},\end{equation}
for all $r>0$ and $x\in \R^n$, then the maximal function
\begin{equation*}\mathcal M_\mu f(x)=\sup_{t> 0}\int_{\R^n} |f(x+ty)|\,d\mu(y)\end{equation*}
satisfies the pointwise bound
\begin{equation*}\label{genmax}\mathcal M_\mu f(x)\leq C[I_1(|\nabla f|)(x)+Mf(x)]\eqsim I_1(|\nabla f|)(x)\end{equation*}
where the equivalence follows from \eqref{pointmax}. The growth condition \eqref{growth} includes the restriction of Lebesgue measure to the unit ball, in which case $\mathcal M_\mu=M$, and the surface measure on the unit sphere, in which case $\mathcal M_\mu=\mathcal M_{\Sn^{n-1}}$.   

Consider the maximal operator $M_\Omega$ for $\Omega \in L^r(\mathbb S^{n-1})$ with $r>1$. By H\"older's inequality we have
\begin{equation}\label{MOmegabdd}M_\Omega f(x)\leq C\|\Omega\|_{L^r(\mathbb S^{n-1})}M_{L^{r'}}f(x)\end{equation}
where, 
$$M_{L^{r'}}f(x)=\sup_{Q\ni x} \left(\dashint_Q |f|^{r'}\right)^{\frac1{r'}}.$$

We are going to see that the exponent $r=n$ plays a special role.  First we claim that the following pointwise bound holds,
\begin{equation}\label{Mnbdd} M_{L^{n'}}f(x)\leq CI_1(|\nabla f|)(x)\end{equation}
where $C$ is a dimensional constant. Indeed, combining 
$$\left(\dashint_Q |f|^{n'}\right)^{\frac1{n'}}\leq \left(\dashint_Q |f-f_Q|^{n'}\right)^{\frac1{n'}}+\dashint_Q|f|.$$
with the classical Poincar\'e-Sobolev inequality,
$$\left(\dashint_Q |f(x)-f_Q|^{n'}\,dx\right)^{\frac1n'}\leq C\,\ell(Q)\left(\dashint_Q|\nabla f(x)|\,dx\right)$$
which holds with a dimensional constant $C$ and cube $Q$ with sidelength $\ell(Q)$, now implies
$$M_{L^{n'}}f(x)\leq CM_1(|\nabla f|)(x)+Mf(x)\leq CI_1(|\nabla f|)(x)$$
by \eqref{pointmax}.  This yields the claimed inequality \eqref{Mnbdd} which is interesting for several reasons. First, it improves the bound for iterated maximal functions \eqref{maxiteratedptwise} because
$$M^kf(x)\leq C_{k} M_{L^{n'}}f(x)$$
which follows from the fact that $M_{L^{n'}}f={M(|f|^{n'})^{\frac1{n'}}}$ is an $A_1$ weight with a constant that only depends on dimension (see \cite{Duobook}).  Second, $n'$ is a critical index in the sense that the pointwise inequality
$$M_{L^{n'+\ep}}f(x)\leq CI_1(|\nabla f|)(x)$$
does \emph{not} hold if $\ep>0$. Indeed, if such an inequality did hold for $\ep>0$, then the boundedness of $I_1:L^p(\R^n)\ra L^{p^*}(\R^n)$ for $1<p<n$ would imply the existence of $g\not=0$ in $C_c^\infty(\R^n)$ such that $Mg\in L^1(\R^n)$, which is impossible. 

The previous inequalities raise the following question: what is the largest maximal function $M_{\mathcal X}$ such that 
$$M_{\mathcal X}f(x)\leq CI_1(|\nabla f|)(x)?$$
This question is open, but we can improve the maximal function $M_{L^{n'}}$ in \eqref{Mnbdd} using Lorentz spaces. The Lorentz space $L^{p,q}(\mu)$, for $0<p<\infty$ and a general measure space $(X,\mu)$, is defined to be the collection of functions that satisfy
$$\|f\|_{L^{p,q}(\mu)}=\left(p\int_0^\infty t^q\mu(\{x\in X:|f(x)|>t\})^{\frac{q}{p}}\frac{dt}{t}\right)^{\frac1q}<\infty,$$
when $0<q<\infty$.  When $q=\infty$, $L^{p,\infty}(\mu)$ is the usual weak $L^p$ space defined by
$$\|f\|_{L^{p,\infty}(\mu)}=\sup_{t>0}t\mu(\{x\in X:|f(x)|>t\})^{\frac1p}<\infty.$$
We will use the normalized Lorentz average associated to a cube $Q$ (or a ball), 
\begin{equation}\label{lorentzavg}\|f\|_{L^{p,q}(Q)}:=\|f\|_{L^{p,q}(Q,\frac{dx}{|Q|})}=\frac{1}{|Q|^{\frac1p}}\|f\mathbf 1_Q\|_{L^{p,q}(\R^n)}.\end{equation}
If we define the Lorentz maximal function
$$M_{L^{p,q}}f(x)=\sup_{Q\ni x}\|f\|_{L^{p,q}(Q)},$$
then we have the following theorem.
\begin{theorem} \label{lorentzmax} If $f\in C_c^\infty(\R^n)$, then
\begin{equation*}M_{L^{n',1}}f(x)\leq CM_1(|\nabla f|)(x)+Mf(x)\end{equation*} 
and hence
\begin{equation}\label{lorentz} M_{L^{n',1}}f(x)\leq CI_1(|\nabla f|)(x).\end{equation}
\end{theorem}
Inequality \eqref{lorentz} implies \eqref{Mnbdd} because
$$M_{L^{n'}}f(x)\leq CM_{L^{n',1}}f(x).$$
by the well-known containment for Lorentz spaces $L^{n',1}\subseteq L^{n',n'}=L^{n'}$ (see Section \ref{proof}). Inequality \eqref{lorentz} follows from the pointwise bounds of $M_1(|\nabla f|)$ and $Mf$ by $I_1(|\nabla f|)$ (see inequality \eqref{pointmax}).

When $\Omega\in L^{n,\infty}(\Sn^{n-1})$, the maximal functions $M_\Omega$ and $M_{L^{n',1}}$ are related via the inequality
$$M_\Omega f(x)\leq  C\|\Omega \|_{L^{n,\infty}(\Sn^{n-1})}M_{L^{n',1}}f(x)$$
which improves \eqref{MOmegabdd} as well. Hence we have the following corollary.
\begin{corollary} \label{MOmega} If $\Omega \in L^{n,\infty}(\Sn^{n-1})$ and $f\in C_c^\infty(\R^n)$ then
\begin{equation*}M_\Omega f(x)\leq C\|\Omega\|_{L^{n,\infty}(\mathbb S^{n-1})}I_1(|\nabla f|)(x).\end{equation*}
\end{corollary}


We have shown that several operators, including specific Calder\'on-Zygmund operators and maximal functions, satisfy a pointwise bound by $I_1(|\nabla f|)$. Our main result concerns a pointwise inequality for maximal rough singular integral operators of the form
$$T^\star_{\Omega} f(x)=\sup_{t>0}\left|\int_{|y|>t}\frac{\Omega(y)}{|y|^{n}}f(x-y)\,dy\right|,$$
when $\Omega$ satisfies appropriate conditions as stated in Theorem \ref{main} below. Rough maximal singular integrals, that is, $T_\Omega^\star$, where $\Omega$ does not have smoothness, are very difficult objects and several open problems remain. For example, the weak (1,1) endpoint for $T^\star_\Omega$ when $\Omega\in L^r(\mathbb S^{n-1})$ for $r>1$ is completely open. The best current results, which are related to the optimality of weighted estimates on $L^2(w)$, are due to Honz\'ik \cite{Hon2}. In \cite{Hon1} Honz\'ik also constructs an example of $\Omega\in L^1(\mathbb S^1)$ such that $T_\Omega$ is bounded on $L^2(\R^2)$, but $T^\star_\Omega$ is not bounded on $L^2(\R^2)$.

Recall that we will always be working in dimension $n\geq 2$. Our main result is the following.



\begin{theorem} \label{main} Suppose $\Omega$ is homogeneous of degree zero and belongs to $L^{n,\infty}(\mathbb S^{n-1})$ with zero average on $\mathbb S^{n-1}$. Then
\begin{equation*}\label{ptwisebdd} T^\star_{\Omega} f(x)\leq C\|\Omega\|_{L^{n,\infty}(\mathbb S^{n-1})} I_1(|\nabla f|)(x), \quad f\in C^\infty_c(\R^n).\end{equation*}
\end{theorem}

\begin{remark} Since $\sigma(\mathbb S^{n-1})<\infty$, when $r\geq n$, we have that 
$$L^r(\mathbb S^{n-1})\subseteq L^n(\mathbb S^{n-1})\subseteq L^{n,\infty}(\mathbb S^{n-1}).$$
So the conclusion of Theorem \ref{main} also holds under the stronger assumption that $\Omega\in L^r(\Sn^{n-1})$ for $r\geq n$. Moreover, assuming $\Omega\in L^r(\mathbb S^{n-1})$ for $r\geq n$ is natural in our setting because of the known weighted estimates for $T_\Omega$ that were used to prove \eqref{onewTOmegaSob}.
\end{remark}



We define the homogeneous weighted Sobolev space $\dot{W}^{1,p}(v)$ to be the closure of $C_c^\infty(\R^n)$ in the semi-norm $\|\nabla f\|_{L^p(v)}$.  As a corollary we have the following.  

\begin{corollary} \label{maintwocor} Suppose $1\leq p,q<\infty$ and $T$ is any one of the following operators: the Hardy-Littlewood maximal operator $M$ or $k$-fold iteration $M^k$; the Lorentz maximal operator $M_{L^{n',1}}$; the spherical maximal operator $\mathcal M_{\Sn^{n-1}}$; a rough maximal operator $M_\Omega$ where $\Omega\in L^{n,\infty}(\Sn^{n-1})$, or a rough singular integral operator $T_\Omega$ or maximal rough singular integral operator $T^\star_\Omega$ where $\Omega\in L^{n,\infty}(\Sn^{n-1})$ and has mean zero.  If $I_1:L^p(v)\ra L^q(u)$, then $$T:\dot W^{1,p}(v) \ra L^q(u)$$ with
\begin{equation}\label{twoweightSISob} \|Tf\|_{L^q(u)}\leq C\|I_1\|_{\mathcal B(L^p(v),L^q(u))}\|\nabla f\|_{L^p(v)}.\end{equation}
Moreover, if $I_1:L^p(v)\ra L^{q,\infty}(u)$, then
$$T:\dot W^{1,p}(v) \ra L^{q,\infty}(u)$$ with
\begin{equation}\label{twoweightweakSISob} \|Tf\|_{L^{q,\infty}(u)}\leq C\|I_1\|_{\mathcal B(L^p(v),L^{q,\infty}(u))}\|\nabla f\|_{L^p(v)}.\end{equation}
\end{corollary}
\begin{remark}  By letting $f=I_1g$, inequality \eqref{twoweightSISob} becomes 
$$\|T(I_1g)\|_{L^q(u)}\leq C\|\mathbf{R}g\|_{L^p(v)}$$
where $\mathbf{R}=(R_1,\ldots,R_n)$ is the vector Riesz transform operator. This control by Riesz transforms is of interest, particularly at the endpoint, and we refer to Schikorra, Spector, and Van Schaftingen \cite{SSV} and the improvement into Lorentz spaces in \cite{Sp} for more on the history and relevance of these inequalities.  
\end{remark}

We remark that the Sobolev mappings \eqref{twoweightSISob} and \eqref{twoweightweakSISob} are new for several of the operators under consideration even in the Lebesgue measure setting. The known endpoint weak inequality for $I_1:L^1(\R^n)\ra L^{n',\infty}(\R^n)$ yields the boundedness 
\begin{equation*}T :\dot W^{1,1}(\R^n)\ra L^{n',\infty}(\R^n),\end{equation*}
for any operator satisfying $|Tf|\leq CI_1(|\nabla f|)$.  For the spherical maximal operator $\mathcal M_{\Sn^{n-1}}$, the inequality
\begin{equation}\label{weakbdMS}\|\mathcal M_{\Sn^{n-1}}f\|_{L^{n',\infty}(\R^n)}\leq C \|\nabla f\|_{L^1(\R^n)}\end{equation}
is new.  Inequality \eqref{weakbdMS} does not follow from the known bounds for the spherical maximal, since $n'=\frac{n}{n-1}$ is a critical index.  If $\mathcal M_{\Sn^{n-1}}$ satisfied the weak type boundedness
\begin{equation}\label{falseweak}\|\mathcal M_{\Sn^{n-1}}g\|_{L^{n',\infty}(\R^n)}\leq C\|g\|_{L^{n'}(\R^n)},\end{equation}
then the Sobolev inequality
$$\|\mathcal M_{\Sn^{n-1}}f\|_{L^{n',\infty}(\R^n)}\leq C \|\nabla f\|_{L^1(\R^n)}$$
would simply follow from the embedding $\dot W^{1,1}(\R^n)\hookrightarrow L^{n'}(\R^n)$. However, the inequality \eqref{falseweak} is false as the example in \cite{Ste} shows (note this example does not belong to $W^{1,1}(\R^n)$).  

There is a similar phenomenon with the operator $M_{L^{n',1}}$. Again, inequality \eqref{lorentz} in Theorem \ref{lorentzmax} and the weak type inequality for $I_1$ imply
$$\|M_{L^{n',1}}f\|_{L^{n',\infty}(\R^n)}\leq C\|\nabla f\|_{L^1(\R^n)}.$$
However, $M_{L^{n',1}}$ is not weak-type $n'$, but rather restricted weak-type, $M_{L^{n',1}}:L^{n',1}(\R^n)\ra L^{n',\infty}(\R^n)$ (see \cite{Per95a}).

Finally, we see that any operator satisfying
$$|Tf(x)|\leq CI_1(|\nabla f|)(x),$$
will have a certain self-improvement.  Indeed, we may use the well-known fact that $A_1$ weights are invariant under powers close to one, that is, 
$$w\in A_1 \Rightarrow w^r\in A_1$$
for some $r>1$.  Given a fixed $f\in C_c^\infty(\R^n)$, since $I_1(|\nabla f|)\in A_1$ we have that $I_1(|\nabla f|)^r\in A_1$ for some $r>1$ and hence
\begin{multline*}|Tf(x)|^r\leq I_1(|\nabla f|)(x)^r \\ \Rightarrow M_{L^r}(Tf)(x)=M(|Tf|^r)(x)^{\frac1r}\leq CI_1(|\nabla f|)(x).\end{multline*}
The value of $r>1$ depends on the $A_1$ constant of $I_1(|\nabla f|)$ which in turn depends on the dimension $n$.  However, we have the following lemma which shows the exponent can range up to the critical index $n'$.

\begin{lemma} \label{A1thm} Suppose $0<\al<n$ and $\mu$ is a finite Borel measure.  If we define 
$$I_\al \mu(x)=\int_{\R^n}\frac{d\mu(y)}{|x-y|^{n-\al}},$$
then $(I_\al\mu)^r\in A_1$ for all $0\leq r<(\frac{n}{\al})'=\frac{n}{n-\al}$.
\end{lemma}

\begin{remark} As far as we can tell, Lemma \ref{A1thm} is new. This result improves the previous fact that $I_\al \mu\in A_1$ because the power can always satisfy $r>1$ independent of $\mu$.  For the fractional maximal operator, the fact that $(M_\al\mu)^r\in A_1$ for $0\leq r<(\frac{n}{\al})'$, was shown by Sawyer \cite{Sawbook}. Also, in \cite{HMPV}, a similar result is needed with very precise bounds, in the case of $M_\al$.
\end{remark}

Thus we have the following self-improvement which follows from Lemma \ref{A1thm}.

\begin{theorem} \label{selfimprove} Suppose $T$ is an operator such that
$$|Tf(x)|\leq CI_1(|\nabla f|)(x), \quad f\in C_c^\infty(\R^n).$$
Then 
\begin{equation}\label{selfimproveineq} M_{L^r}(Tf)(x)\leq CI_1(|\nabla f|)(x), \quad f\in C_c^\infty(\R^n)\end{equation}
for all $0\leq r<n'.$
\end{theorem}

\begin{remark} One may wish to compare the result of Corollary \ref{selfimprove} with inequality \eqref{Mnbdd} which holds for $r=n'$. Inequality \eqref{Mnbdd} can be seen as inequality \eqref{selfimproveineq} for $r=n'$ and the identity operator $T=\mathsf{Id}$.  Of course, inequality \eqref{Mnbdd} is proved using different techniques that use the fact that the identity  is a nice operator. In general we cannot expect inequality \eqref{selfimproveineq} to hold when $r=n'$.
\end{remark}

In Section \ref{proof} we present the proof of our main results including Theorems \ref{lorentzmax} and \ref{main}. Our main tool is the local Poincar\'e-Sobolev inequality in Lorentz spaces and we will also present some needed preliminaries.  In Section \ref{oneweight} we prove some new consequences of our pointwise inequalities in the one weight setting.  Section \ref{twoweight} contains a discussion of two weight inequalities for fractional integrals, which imply inequality \eqref{twoweightSISob}.  We obtain some new weak-type endpoint inequalities  in Section \ref{endpt}. 

\section{Proofs of our main results} \label{proof} 
We will now prove Theorems \ref{lorentzmax} and \ref{main} as well as Lemma \ref{A1thm}.  We begin with some preliminaries.  Given a function $f$ and measurable set $E$ with $0<|E|<\infty$ define
$$f_E=\dashint_E f=\frac{1}{|E|}\int_E f(x)\,dx.$$
To prove our results we will  use a Poincar\'e-Sobolev inequality in the scale of Lorentz spaces.  We recall some basic fact about Lorentz spaces that will be needed.  First, the Lorentz spaces are nested in $q$, namely, if $0<q_1< q_2\leq \infty$ then
$$\|f\|_{L^{p,q_2}(\mu)}\leq C\|f\|_{L^{p,q_1}(\mu)}$$
and hence $L^{p,q_1}(\mu)\subseteq L^{p,q_2}(\mu)$.  Second, we also have the following version of H\"older's inequality
$$\int_X|fg|\,d\mu\leq C \|f\|_{L^{p,q}(\mu)}\|g\|_{L^{p',q'}(\mu)}$$
which holds for all $1<p<\infty$ and $1\leq q\leq \infty$.

The classical Poincar\'e-Sobolev inequality 
\begin{equation}\label{ClassPoin}\left(\dashint_B |f(x)-f_B|^q\,dx\right)^{\frac1q}\leq C\,r(B)\left(\dashint_B|\nabla f(x)|^p\,dx\right)^{\frac1p}\end{equation}
holds for all balls $B$ with radius $r(B)$, $1\leq p<n$, $1\leq q\leq p^*=\frac{np}{n-p}$ and $f\in C_c^\infty(\R^n)$. Inequality  \eqref{ClassPoin} also holds for all cubes if the radius of the ball is replaced by sidelength of the cube. More on Poincar\'e and Poincar\'e-Sobolev inequalities is contained in the books \cite[Chapter 7]{GT} or \cite[Chapter 3]{KLV}.  We also mention the work of the third author and Rela \cite{PR} for several extensions of inequality \eqref{ClassPoin}.  

In our context we will need the following improvement of the Poincar\'e-Sobolev inequality in Lorentz spaces due to O'Neil \cite{ON} and Peetre \cite{Pee} (see also \cite{MP}). They showed that $f\in C^\infty_c(\R^n)$, then for any ball $B$
\begin{equation*}\label{rspoin}\|f-f_B\|_{L^{p^*,p}(B)}\leq Cr(B)\left(\dashint_B|\nabla f|^p\right)^{\frac1p},\end{equation*}
for $1\leq p<n$ where we are using the normalized Lorentz average.  In particular, when $p=1$, we have the following 
\begin{equation}\label{lorentzPoin}\|f-f_B\|_{L^{n',1}(B)}\leq C r(B)\dashint_B |\nabla f|\end{equation}
and again we remark that inequality \eqref{lorentzPoin} holds if we replace balls by cubes. 

\begin{proof}[Proof of Theorem \ref{lorentzmax}]  Fix $x\in \R^n$ and let $Q$ be a cube containing $x$. Since $\|\cdot\|_{L^{n',1}(Q)}$ is a norm 
\begin{align*}\|f\|_{L^{n',1}(Q)}\leq \|f-f_Q\|_{L^{n',1}(Q)}+\|f_Q\|_{L^{n',1}(Q)}\leq C\ell(Q)\dashint_Q|\nabla f|+|f_Q|
\end{align*}
where we have used the Poincar\'e inequality \eqref{lorentzPoin}. Taking the supremum over $Q$ containing $x$ yields the desired inequality
$$M_{L^{n',1}}f(x)\leq CM_1(|\nabla f|)(x)+Mf(x).$$
\end{proof}


\begin{proof}[Proof of Theorem \ref{main}] For $t>0$ let
$$T^t_\Omega f(x)=\int_{|y|>t}\frac{\Omega(y)}{|y|^n}f(x-y)\,dy$$
so that $T^\star_\Omega f=\sup_{t>0}|T^t_\Omega f|$.  Let $f\in C^\infty_c(\R^n)$, and pick $k_0$ so that $2^{k_0-2}<t\leq 2^{k_0-1}$ consider the decomposition into dyadic annuli
\begin{multline*}T^t_{\Omega} f(x)=\int_{t<|y|\leq2^{k_0-1}} \frac{\Omega(y)}{|y|^{n}}f(x-y)\,dy+\sum^\infty_{k=k_0} \int_{2^{k-1}<|y|\leq 2^{k}} \frac{\Omega(y)}{|y|^{n}}f(x-y)\,dy\\ =\int_{t<|y|\leq2^{k_0-1}}\frac{\Omega(y)}{|y|^{n}}(f(x-y)-c_{k_0-1})\,dy\\ +\sum_{k=k_0}^\infty \int_{2^{k-1}<|y|\leq 2^{k}}\frac{\Omega(y)}{|y|^{n}}(f(x-y)-c_k)\,dy,\end{multline*}
where we have used the zero average of $\Omega$ on $\Sn^{n-1}$ to add a constant $c_k$ to be chosen later.  We now have
\begin{align}
\lefteqn{T^\star_{\Omega} f(x)\leq \sum_{k\in \Z} \int_{2^{k-1}<|y|\leq 2^{k}} \frac{|\Omega(y)|}{|y|^{n}}|f(x-y)-c_k|\,dy}\nonumber\\
&\leq C\sum_{k\in \Z} \frac{1}{2^{kn}}\int_{|y|\leq 2^{k}} {|\Omega(y)|}|f(x-y)-c_k|\,dy\nonumber\\
&\leq C\sum_{k\in \Z} \|\Omega\|_{L^{n,\infty}(B(0,2^k))}\|f(x-\cdot)-c_k\|_{L^{n',1}(B(0,2^k))} \label{leftoff}
\end{align}
where we have used the Lorentz H\"older's inequality in the third line and the normalized Lorentz average as defined in \eqref{lorentzavg}.  Notice that
\begin{align*}|\{y\in B(0,2^k):|\Omega(y)|>\lambda\}|&=\int_{\R^n}\mathbf 1_{\{y\in B(0,2^k):|\Omega(y)|>\lambda\}}(y)\,dy \\
&=\int_0^{2^k}\int_{\mathbb S^{n-1}} \mathbf 1_{\{y'\in \Sn^{n-1}:|\Omega(y')|>\lambda\}}(y')r^{n-1}\,drd\sigma(y')\\
&=\frac{2^{kn}}{n}\sigma(\{y'\in \Sn^{n-1}:|\Omega(y')|>\lambda\}).
\end{align*}
Hence 
$$\|\Omega\|_{L^{n,\infty}(B(0,2^k))}=c_n\|\Omega\|_{L^{n,\infty}(\Sn^{n-1})}.$$
If we let $B_k=B(x,2^{k})$ and $c_k=f_{B_k}=\dashint_{B_k} f$, then we have
$$\|f(x-\cdot)-c_k\|_{L^{n',1}(B(0,2^k))}=\|f-f_{B_k}\|_{L^{n',1}(B_k)}.$$
Using these calculations in the sum \eqref{leftoff} we obtain 
\begin{equation*}T^\star_{\Omega} f(x)\leq C\|\Omega\|_{L^{n,\infty}(\mathbb{S}^{n-1})} \sum_{k\in \Z} \|f-f_{B_k}\|_{L^{n',1}(B_k)}\end{equation*}
where, again, we have the normalized Lorentz average.  The Lorentz Poincar\'e-Sobolev inequality \eqref{lorentzPoin} now gives
$$T^\star_{\Omega} f(x)\leq C\|\Omega\|_{L^{n,\infty}(\mathbb{S}^{n-1})}\sum_{k\in \Z}r(B_k)\dashint_{B_k}|\nabla f(y)|\,dy.$$
To finish the proof we will show that the sum on the right hand side of the inequality is bounded by $I_1(|\nabla f|)(x)$.  First notice that the sum converges for $f\in C^\infty_c(\R^n)$. Indeed, there exists a $K$ such that $\mathsf{supp} f\subseteq B_{K}$.  Hence,
\begin{multline*}\sum_{k\in \Z}r(B_k)\dashint_{B_k}|\nabla f(y)|\,dy=\sum_{k=-\infty}^{K-1}r(B_k)\dashint_{B_k}|\nabla f(y)|\,dy\\ +\sum_{k=K}^{\infty}r(B_k)\dashint_{B_k}|\nabla f(y)|\,dy.\end{multline*}
Since $|\nabla f|\leq C\chi_{B_{K}}$, each of these sums converge.  Indeed
$$\sum_{k=-\infty}^{K-1}r(B_k)\dashint_{B_k}|\nabla f(y)|\,dy\leq C\sum_{k=-\infty}^{K-1} 2^{k }<\infty$$
and
$$\sum_{k=K}^{\infty}r(B_k)\dashint_{B_k}|\nabla f(y)|\,dy\leq C|B_{K}|\sum_{k=K}^\infty 2^{k(1-n)}<\infty.$$
Next we have
\begin{align*}\lefteqn{\sum_{k\in \Z} r(B_k)\dashint_{B_k}|\nabla f(y)|\,dy=\sum_{k\in \Z}\frac{1}{\omega_n2^{k(n-1)}}\int_{|x-y|\leq2^k}|\nabla f(y)|\,dy}\\
&=\frac1{\omega_n}\sum_{k\in \Z}\frac{1}{2^{k(n-1)}}\int_{2^{k-1}<|x-y|\leq2^k}|\nabla f(y)|\,dy\\ 
&\qquad\qquad +\frac{1}{\omega_n}\sum_{k\in \Z}\frac{1}{2^{k(n-1)}}\int_{|x-y|\leq 2^{k-1}}|\nabla f(y)|\,dy\\
&\leq \frac1{\omega_n}\sum_{k\in \Z}\int_{2^{k-1}<|x-y|\leq2^k}\frac{|\nabla f(y)|}{|x-y|^{n-1}}\,dy\\
&\qquad\qquad +2^{1-n}\sum_{k\in \Z}r(B_{k-1})\dashint_{B_{k-1}}|\nabla f(y)|\,dy\\
&\leq \frac1{\omega_n}I_1(|\nabla f|)(x)+2^{1-n}\sum_{k\in \Z} r(B_k)\dashint_{B_k}|\nabla f(y)|\,dy.
\end{align*}
The factor $2^{1-n}<1$ and by rearranging this inequality we have
$$\sum_{k\in \Z} r(B_k)\dashint_{B_k}|\nabla f(y)|\,dy\leq \frac{1}{\omega_n(1-2^{1-n})}I_1(|\nabla f|)(x)$$
thus completing our proof.
\end{proof}
\begin{remark} It is possible to avoid the absorption argument in the proof by instead using a Poincar\'e inequality on annuli $A(x,r)=B(x,2r)\backslash B(x,r)$ since it is a particular case of a John domain. 
\end{remark}
We end this section with the proof of Lemma \ref{A1thm}.
\begin{proof}[Proof of Lemma \ref{A1thm}] Fix a cube $Q$, $x\in Q$, and let $0< r< n/(n-\al)$. We will show that
$$\left(\dashint_Q(I_\al \mu)^r\right)^{\frac1r}\leq C(I_\al \mu)(x).$$ 
Separate 
$$I_\al\mu(u)=\int_{\R^n}\frac{\mathbf 1_{2Q}(y)}{|u-y|^{n-\al}}\,d\mu(y)+\int_{\R^n}\frac{\mathbf 1_{(2Q)^c}(y)}{|u-y|^{n-\al}}\,d\mu(y).$$
The operator $I_\al$ is bounded from $L^1(\R^n)$ to $L^{(\frac{n}{\al})',\infty}(\R^n)$ and this can be extended to measures $\mu(\R^n)<\infty$, in particular
$$\|I_\al \mu\|_{L^{(\frac{n}{\al})',\infty}(\R^n)}\leq C\mu(\R^n).$$
Now for the local part, by Kolmogorov's inequality (or the normability of the weak-type spaces) since $r<(\frac{n}{\al})'$
\begin{align*}\left(\dashint_QI_\al(\mathbf 1_{2Q}\mu)^r\right)^{\frac1r}&\leq C|Q|^{-\frac{1}{(n/\al)'}}\|I_\al(\mathbf 1_{2Q}\mu)\|_{L^{(\frac{n}{\al})',\infty}(\R^n)}\\ 
&= C|Q|^{\frac{\al}{n}-1}\|I_\al(\mathbf 1_{2Q}\mu)\|_{L^{(\frac{n}{\al})',\infty}(\R^n)} \\
&\leq C\frac{1}{\ell(Q)^{n-\al}}\int_{2Q}\,d\mu(y)\\ 
&\leq C\int_{2Q}\frac{d\mu(y)}{|x-y|^{n-\al}} \\
&\leq CI_\al\mu(x).\end{align*}
For the non-local notice that for any $u\in Q$ and $y\in (2Q)^c$ we have
$$|x-y|\eqsim |u-y|$$
and in particular
$$\frac{1}{|u-y|^{n-\al}}\leq C \frac{1}{|x-y|^{n-\al}}, \quad u\in Q, y\in (2Q)^c.$$
Hence
$$I_\al(\mathbf 1_{(2Q)^c}\mu)(u)\leq C I_\al(\mathbf 1_{(2Q)^c}\mu)(x)\leq C(I_\al \mu)(x), \quad \forall u\in Q$$
and this implies
$$\left(\dashint_QI_\al(\mathbf 1_{(2Q)^c}\mu)^r\right)^{\frac1r}\leq C (I_\al\mu)(x).$$
In particular by combining both estimates we have
\begin{equation*}\left(\dashint_Q (I_\al\mu)^r\right)^{\frac1r}\leq \left(\dashint_QI_\al(\mathbf 1_{2Q}\mu)^r\right)^{\frac1r}+\left(\dashint_QI_\al(\mathbf 1_{(2Q)^c}\mu)^r\right)^{\frac1r} \leq C I_\al \mu(x)\end{equation*}
which completes the proof that $(I_\al \mu)^r$ is an $A_1$ weight.
\end{proof}

\section{One weight consequences} \label{oneweight}
In this section we study dependence on the constant $[w]_{A_{p,p^*}}$ in the inequality
\begin{equation}\label{depend}\|wTf\|_{L^{p^*}(\R^n)}\leq C_{[w]_{A_{p,p^*}}}\|w\nabla f\|_{L^p(\R^n)}.\end{equation}
First we will  discuss this dependence using previously known techniques for the individual operators. Specifically, we will consider the best bounds in the inequality $$T:L^{p^*}(w^{p^*})\ra L^{p^*}(w^{p^*}),$$ which are known or can be readily calculated for some operators. The best dependence in the embedding $\dot W^{1,p}(w^p)\hookrightarrow L^{p^*}(w^{p^*})$ is also known and by combining the two inequalities we obtain a bound on $C_{[w]_{A_{p,p^*}}}$.  This naive approach actually yields new results. Of course these calculations are futile when compared to the results that follow from our pointwise bounds, but we present them so the reader can fully appreciate the power of our pointwise bounds. We end the section with the improved bounds and a new Coifman-Fefferman inequality.

\subsection{Weighted estimates from previously known sharp bounds}  It is well-known that the Hardy-Littlewood maximal operator and Calder\'on-Zygmund operators are bounded on $L^p(w)$ for $1<p<\infty$ when the weight $w$ belongs to the $A_p$ class
$$[w]_{A_p}=\sup_Q \left(\dashint_Q w\right) \left(\dashint_Q w^{1-p'}\right)^{p-1}<\infty.$$

Meanwhile, as mentioned in the introduction the Riesz potential, $I_1$, and fractional maximal function, $M_1$, map $L^p(w^p)\ra L^{p^*}(w^{p^*})$ precisely when $1<p<n$, $\frac1{p^*}=\frac1p-\frac{1}{n}$, and $w\in A_{p,p^*}$
$$[w]_{A_{p,p^*}}=\sup_Q \left(\dashint_Q w^{p^*}\right) \left(\dashint_Q w^{-p'}\right)^{\frac{p^*}{p'}}<\infty.$$
Notice that $w\in A_{p,p^*}$ if and only if $w^{p^*}\in A_{1+\frac{p^*}{p'}}=A_{\frac{p^*}{n'}}$ and $[w]_{A_{p,p^*}}=[w^{p^*}]_{A_{\frac{p^*}{n'}}}$.  

For other operators under consideration, the story is significantly more complicated. For rough singular integrals, recall that if $\Omega\in L^r(\mathbb S^{n-1})$ and $p>r'$ then  
\begin{equation}\label{roughDuo} T_\Omega:L^p(w)\ra L^p(w), \quad \ \ \ w\in A_{\frac{p}{r'}}.\end{equation}
As far as we know, no weighted estimates for $T_\Omega$ when $\Omega \in L^{r,\infty}(\Sn^{n-1})$ have been considered. Weighted estimates for the spherical maximal function $\mathcal M_{\Sn^{n-1}}$ are also arduous to prove. They were investigated by Cowling, Garc\'ia-Cuerva, and Gunawan \cite{CGG} and Lacey \cite{Lac}. The actual conditions on the weights are too obscure to state here, but we do remark that some quantitative bounds for $\mathcal M_{\Sn^{n-1}}$ can be found in \cite{Lac}.



The quantitative dependence on weighted constants is a well studied area and the best possible bounds are known for many operators.  Buckley \cite{Buck} showed that the sharp dependence for the Hardy-Littlewood maximal operator is
\begin{equation}\label{buckbdd} \|Mf\|_{L^p(w)}\leq C[w]_{A_p}^{\frac{p'}{p}}\|f\|_{L^p(w)}.\end{equation}
The sharp bound for general Calder\'on-Zygmund operators
\begin{equation}\label{A2}\|Tf\|_{L^p(w)}\leq C[w]_{A_p}^{\max\{1,\frac{p'}{p}\}}\|f\|_{L^p(w)}, \qquad 1<p<\infty\end{equation}
 is due to Hyt\"onen \cite{Hyt1}. Meanwhile the sharp bound for $I_1$ was shown in \cite{LMPT} to be
\begin{equation}\label{fracint}\|wI_1 f\|_{L^{p^*}(\R^n)}\leq C[w]_{A_{p,p^*}}^{\frac{1}{n'}\max\{1,\frac{p'}{p^*}\}} \|wf\|_{L^p(\R^n)}.\end{equation}
The sharp dependence for the weighted Sobolev inequality \eqref{oneweightSob} is given by
\begin{equation}\label{sharpSob} \|wf\|_{L^{p^*}(\R^n)}\leq C[w]_{A_{p,{p^*}}}^{\frac1{n'}}\|w\nabla f\|_{L^p(\R^n)}\end{equation}
which is better than \eqref{fracint} in the range $1<p<\frac{2n}{n+1}$. The bound \eqref{sharpSob} can be seen from the truncation method (see \cite{LMPT}) and the example showing sharpness is contained in \cite{CM}.  

Quantitative bounds for rough singular integrals are much more difficult and few, if any, are known to be sharp. The best results in our context are due to Conde-Alonso, Culiuc, Di-Plinio, and Ou \cite{CCDPO}; Li, P\'erez, Rivera-R\'ios, and Roncal \cite{LPRR}; and Hyt\"onen, Roncal, Tapiola \cite{HRT}. In \cite{CCDPO} it is shown that if $\Omega\in L^r(\mathbb{S}^{n-1})$ and $q>r'$ then
\begin{equation} \label{roughconst} \|T_\Omega  f\|_{L^q(w)}\leq C[w]_{A_{\frac{q}{r'}}}^{\max\{1,\frac{1}{q-r'}\}} \|f\|_{L^q(w)}, \quad w\in A_{\frac{q}{r'}}.\end{equation}

Now consider the original problem of finding the quantitative dependence in \eqref{depend} for $T=M^k$ or $T=T_\Omega$. For the quantitative dependence in the inequality
$$\|wM^kf\|_{L^{p^*}(\R^n)}\leq C_{[w]_{A_{p,p^*}}}\|w\nabla f\|_{L^p(\R^n)}$$
we can first iterate inequality \eqref{buckbdd} to get
$$\|M^kf\|_{L^p(w)}\leq C[w]_{A_p}^{k\frac{p'}{p}}\|f\|_{L^p(w)}$$
which is actually sharp (see \cite{LPR}). However, this is less than optimal since $w^{p^*}$ belongs to the smaller class $A_{\frac{p^*}{n'}}(\subsetneq A_{p^*})$.  Instead we can use the following bound of Duoandikoetxea \cite{Duo2}: 
\begin{equation} \label{duobd} \|Mf\|_{L^q(w)}\leq C[w]_{A_s}^{\frac1q} \|f\|_{L^q(w)} \quad w\in A_s, \ \ 1\leq s<q.\end{equation}
Combining the bounds \eqref{duobd} with $s=\frac{p^*}{n'}$ and \eqref{sharpSob} yields
\begin{equation}\label{baditeratedM}\|wM^kf\|_{L^{p^*}(\R^n)} \leq C[w]_{A_{p,p^*}}^{\frac{k}{p^*}+\frac1{n'}}\|w\nabla f\|_{L^p(\R^n)},\end{equation}
which is new, but can be significantly improved.  

For rough singular integrals if we combine the bounds \eqref{sharpSob} and \eqref{roughconst} we obtain the following quantitative version of \eqref{onewTOmegaSob} when $\Omega\in L^r(\Sn^{n-1})$ with $r\geq n$
\begin{equation}\label{badbdd}\|wT_\Omega f\|_{L^{p^*}(\R^n)}\leq C[w]_{A_{p,{p^*}}}^{\max\{1,\frac{1}{{p^*}-n'}\}+\frac1{n'}}\|w\nabla f\|_{L^p(\R^n)}.\end{equation}

We may improve the bounds even more using $A_\infty$ constants.  Define $A_\infty=\bigcup_{p>1}A_p$.  A weight $w\in A_\infty$ can be quantified with the Fuji-Wilson constant
$$[w]_{A_\infty}=\sup_Q\frac1{w(Q)}\int_Q M({\mathbf 1}_Qw),$$
introduced in  Hyt\"onen and the third author \cite{HyPe}.
For bounds in terms of the constant $[w]_{A_\infty}$ it was shown, first in the case $p=2$ in \cite{HyPe} and for general $1<p<\infty$ in \cite{HyL, HyLP}, that
\begin{equation}\label{mixedCZO} \|Tf\|_{L^p(w)}\leq C[w]_{A_p}^{\frac1{p}}([w]_{A_\infty}^{\frac{1}{p'}}+[w^{1-p'}]_{A_\infty}^{\frac1p})\|f\|_{L^p(w)}.\end{equation}
Inequality \eqref{mixedCZO} is an improvement over inequality \eqref{A2} because
\begin{equation*}[w]_{A_p}^{\frac1{p}}([w]_{A_\infty}^{\frac{1}{p'}}+[w^{1-p'}]_{A_\infty}^{\frac1p})\leq[w]_{A_p}+[w^{1-p'}]_{A_{p'}}\eqsim [w]_{A_p}^{\max\{1,\frac{p'}{p}\}}.\end{equation*}
Cruz-Uribe and the second author \cite{CM2} extended this to the fractional integral case showing
\begin{equation}\label{mixedIal}\|wI_1 f\|_{L^{p^*}(\R^n)} \leq C[w]^{\frac1{p^*}}_{A_{p,p^*}}([w^{p^*}]_{A_\infty}^{\frac1{p'}}+[w^{-p'}]_{A_\infty}^{\frac1{p^*}})\| wf\|_{L^p(\R^n)}.\end{equation}
(see also \cite{R} for a particular situation of this inequality). Again this improves \eqref{fracint} because
$$ [w]^{\frac1{p^*}}_{A_{p,p^*}}([w^{p^*}]_{A_\infty}^{\frac1{p'}}+[w^{-p'}]_{A_\infty}^{\frac1{p^*}})\leq [w]^{\frac1{n'}}_{A_{p,p^*}}+[w]_{A_{p,p^*}}^{\frac{1}{n'}\frac{p'}{p^*}}\eqsim [w]_{A_{p,p^*}}^{\frac{1}{n'}\max\{1,\frac{p'}{p^*}\}}.$$
For rough singular integrals with $\Omega\in L^\infty(\mathbb S^{n-1})$ the only known such mixed $A_p$-$A_\infty$ result can be found in \cite{LPRR}: 
\begin{equation*}\label{roughmixed} \|T_\Omega f\|_{L^p(w)}\leq C[w]_{A_p}^{\frac1{p}}([w]_{A_\infty}^{\frac{1}{p'}}+[w^{1-p'}]_{A_\infty}^{\frac1p})\min\{[w]_{A_\infty},[w^{1-p'}]_{A_\infty}\}\|f\|_{L^p(w)}.\end{equation*}
We will not further pursue combining mixed $A_p$-$A_\infty$ weighted bounds for the Sobolev inequality.

\subsection{New bounds}
As a consequence of Corollary \ref{maintwocor} and inequality \eqref{mixedIal} we have the following.
\begin{corollary} \label{nicebdd} Suppose $T$ is any one of the operators mentioned in the statement of Corollary \ref{maintwocor}, $1<p<n$, and $p^*=\frac{np}{n-p}$. If $w\in A_{p,{p^*}}$, then
$$\|wTf\|_{L^{{p^*}}(\R^n)}\leq C[w]^{\frac1{p^*}}_{A_{p,{p^*}}}([w^{p^*}]_{A_\infty}^{\frac1{p'}}+[w^{-p'}]_{A_\infty}^{\frac1{p^*}})\|w\nabla f\|_{L^p(\R^n)}.$$
As a consequence, we have
\begin{equation*}\label{newgoodbdd}\|wTf\|_{L^{{p^*}}(\R^n)}\leq C[w]_{A_{p,{p^*}}}^{\frac1{n'}\max\{1,\frac{p'}{{p^*}}\}}\|w\nabla f\|_{L^p(\R^n)}.\end{equation*}
\end{corollary}
\begin{remark} 
The quantitative dependence in Corollary \ref{nicebdd} is significantly better than the bounds from inequality \eqref{baditeratedM} for $M^k$ and inequality \eqref{badbdd} for $T_\Omega$.  For the spherical maximal operator, $\mathcal M_{\Sn^{n-1}}$, these estimates seem to be completely new in both the class of weights and the quantitative bounds. Likewise, the class of weights and bounds for $T_\Omega$ and $T^\star_\Omega$ are new under the assumption $\Omega\in L^{n,\infty}(\Sn^{n-1})$. Previously, weighted inequalities for $T_\Omega$ when $w^{p^*}\in A_{\frac{p^*}{n'}}$ are only known to hold when $\Omega\in L^r(\Sn^{n-1})$ for $r\geq n$. We also point out that the iterated maximal function bound has the form
$$\|wM^kf\|_{L^{{p^*}}(\R^n)}\leq C_k[w]_{A_{p,{p^*}}}^{\frac1{n'}\max\{1,\frac{p'}{{p^*}}\}}\|w\nabla f\|_{L^p(\R^n)}.$$
The constant $C_k$ will blow up as $k$ becomes large, but the dependence on $[w]_{A_{p,p^*}}$ is fixed for all $k$.  This  phenomenon is new. We do not know of any examples showing sharpness of these inequalities.
\end{remark}

We end this section with the following version of \eqref{CoifFeff} for fractional integrals due to 
Muckenhoupt and Wheeden \cite{MW},  
\begin{equation} \label{fraccontrol} \|I_\al f\|_{L^p(w)}\leq C\|M_\al f\|_{L^p(w)}, \quad w\in A_\infty, \ 0<p<\infty.  \end{equation}
This inequality was improved in  \cite{PW, CLPR} for a larger class of weights in the case $p\in(1,\infty)$. By using Theorem \ref{main} combined with inequality \eqref{fraccontrol} we obtain the following Coifman-Fefferman type estimate involving the gradient. In particular for rough singular integrals with $\Omega\in L^{n,\infty}(\Sn^{n-1})$ we have inequality \eqref{CFgrad}.

\begin{corollary} \label{weightcontrol} Suppose $T$ is any one of the operators mentioned in the statement of Corollary \ref{maintwocor}.  If $0<p<\infty$ and $w\in A_\infty$, then
$$\|Tf\|_{L^p(w)}\leq C\|M_1(|\nabla f|)\|_{L^p(w)}.$$
\end{corollary}

\section{Two weight consequences}\label{twoweight}

The general two weight Sobolev inequality 
\begin{equation}\label{twoSob} \|f\|_{L^q(u)}\leq C\|\nabla f\|_{L^p(v)}\end{equation}
has a long history and has been studied by several mathematicians (see  \cite{CWW, Fef, KerSaw, LN, Per95, Maz}). For example, the inequality
\begin{equation*}\label{uncert} \left(\int_{\R^3}\frac{|f(x)|^2}{|x|^2}\,dx\right)^{\frac12}\leq C\left(\int_{\R^3}{|\nabla f(x)|^2}\,dx\right)^{\frac12}\end{equation*}
is related to the uncertainty principle in mathematical physics (see Reed and Simon \cite[p.169]{RS}).  Again, \eqref{twoSob} follows from the pointwise bound \eqref{represent} and the boundedness of the Riesz potential operator $I_1:L^p(v)\ra L^q(u)$.   In fact, when $p\leq q$, inequality \eqref{twoSob} can be derived from the weak type boundedness $I_1:L^p(v)\ra L^{q,\infty}(u)$ via the truncation method of Maz'ya \cite{Maz} (see also \cite{Haj,LMPT,LN,PR}).  In particular, the inequality
\begin{equation}\label{weakSob} \|f\|_{L^q(u)}\leq C\|I_1\|_{\mathcal B(L^p(v),L^{q,\infty}(u))}\|\nabla f\|_{L^p(v)}\end{equation}
holds. In our case, inequality \eqref{twoSob} extends to 
\begin{equation*}\label{strongtwoweight}\|Tf\|_{L^q(u)}\leq C\|\nabla f\|_{L^p(v)}\end{equation*}
for any operator that satisfies $|Tf|\leq CI_1(|\nabla f|)$ provided the strong boundedness $I_1:L^p(v)\ra L^q(u)$ holds (Corollary \ref{maintwocor}).  The truncation method of Maz'ya cannot be applied when there is a general operator on the lefthand side of the inequality. Specifically, the following inequality
$$\|Tf\|_{L^{q}(u)}\leq C\|I_1\|_{\mathcal B(L^p(v),L^{q,\infty}(u))}\|\nabla f\|_{L^p(v)},$$
may not hold for a general operator satisfying $|Tf|\leq CI_1(|\nabla f|)$. Indeed, the operator $T=M_{L^{n'}}$ satisfies the weak inequality
$$\|M_{L^{n'}}f\|_{L^{n',\infty}(\R^n)}\leq C\|\nabla f\|_{L^1(\R^n)}$$
but does not satisfy the strong inequality.

The two weight boundedness of $$I_\al:L^p(v)\ra L^q(u)$$ is well studied.  This boundedness depends heavily on the relationship between $p$ and $q$ and splits into two cases: the strict upper triangular case $q<p$ and the lower triangular case $p\leq q$.  We point out the work of H\"anninen, Hyt\"onen, and Li \cite{HHL} and Vuorinen \cite{Vuo} who independently found unified characterizations for dyadic positive operators, related to $I_\al$, which hold for all $1<p,q<\infty$. The strict upper triangular case $q<p$ is the most difficult, and we mention the work of Verbitsky and others \cite{COV2,Ver}.  In the dyadic setting there are also characterizations in the case $q<p$ given by Tanaka \cite{Tan}.   

For the lower triangular case, i.e., $p\leq q$, a complete characterization of the weights such that $I_\al:L^p(v)\ra L^q(u)$ is known due to Sawyer \cite{Saw1} and \cite{Saw2}.  Sawyer showed that the Riesz potential is a bounded linear operator from $L^p(v)$ to $L^q(u)$ if and only if two testing conditions for $I_\al$ hold.
We also refer readers to the unpublished manuscript by Lacey, Sawyer, and Uriarte-Tuero \cite{LSU} and the survey article by Hyt\"onen \cite{Hyt2} for different proofs of the two weight testing characterizations for general dyadic operators. In our setting we have the following corollary.

\begin{corollary}  Suppose $T$ is any of the operators mentioned in statement of Corollary \ref{maintwocor}.  If $1<p\leq q<\infty$, and $(u,v)$ satisfy the testing conditions
\begin{equation*} \left(\int_QI_1({\mathbf 1}_Q v^{1-p'})^{q} u\right)^{\frac1{q}}\leq C\left(\int_Qv^{1-p'}\right)^\frac{1}{p}\end{equation*} 
and
\begin{equation*} \left(\int_QI_1({\mathbf 1}_Qu)^{p'} v^{1-p'}\right)^{\frac1{p'}}\leq C\left(\int_Q u\right)^\frac{1}{q'},\end{equation*} 
then
$$\|Tf\|_{L^q(u)}\leq C\|\nabla f\|_{L^p(v)}.$$
\end{corollary}

While testing conditions are necessary and sufficient for boundedness, they are difficult to check in practice and depend on the operator itself. In the lower triangular case there are universal conditions that imply $I_\al:L^p(v)\ra L^q(u)$ which are sufficient and sharp.  The third author initiated the study of the so-called ``bump conditions'' which we describe below. A necessary, but not sufficient, condition for $I_\al:L^p(v)\ra L^{q}(u)$, is the following two weight $A^\al_{p,q}$ condition:
\begin{equation}\label{twoAp} \sup_Q |Q|^{\frac\al{n}+\frac1q-\frac1p}\left(\dashint_Qu\right)^{\frac1q}\left(\dashint_Qv^{1-p'}\right)^{\frac1{p'}}<\infty.\end{equation}
Sufficiently enlarging the averages in \eqref{twoAp} will imply the boundedness of $I_\al$. To understand this enlarging we require some background on Orlicz spaces. Let $\Theta$ be a Young function and define the Orlicz average of a measurable function to be 
$$\|f\|_{L^\Theta(Q)}=\inf\left\{\lambda>0:\dashint_Q\Theta\Big(\frac{|f|}{\lambda}\Big)\leq 1\right\}.$$
Notice that when $\Theta(t)=t^p$ we have that $\|\cdot\|_{L^\Theta(Q)}=\|\cdot\|_{L^p(Q)}$ is the $L^p$ average over $Q$ and the necessary condition \eqref{twoAp} can be written as
$$\sup_Q |Q|^{\frac\al{n}+\frac1q-\frac1p}\|u^{\frac1q}\|_{L^q(Q)}\|v^{-\frac1p}\|_{L^{p'}(Q)}<\infty.$$
The idea of bump conditions is to use the theory of Orlicz spaces to increase the $L^q$ and $L^{p'}$ averages as little as possible to obtain a sufficient condition.  Every Young function $\Theta$ has an associate function $\bar{\Theta}$ such that $\Theta^{-1}(t)\bar{\Theta}^{-1}(t)\eqsim t$. If $\Theta(t)=t^p$ then we have $\bar{\Theta}(t)\eqsim t^{p'}$.  Given $1<p<\infty$, we say $\Theta\in B_p$ if
$$\int_1^\infty \frac{\Theta(t)}{t^p}\frac{dt}{t}<\infty.$$
The third author introduced the $B_p$ integrability condition and used it to find sufficient bump conditions for $I_\al$ and more general potential operators in \cite{Per95a,Per94}.

\begin{theorem}[\cite{Per94}]\label{Perbump} Suppose $1<p\leq q<\infty$, and $(\Phi,\Psi)$ is a pair of Young functions. If $\bar{\Phi} \in B_{q'}$, $\bar{\Psi}\in B_p$ and $(u,v)$ satisfies 
\begin{equation}\label{Apqbump} \sup_Q |Q|^{\frac\al{n}+\frac1q-\frac1p}\|u^{\frac1q}\|_{L^\Phi(Q)}\|v^{-\frac1p}\|_{L^\Psi(Q)}<\infty,\end{equation}
then $I_\al:L^p(v)\ra L^q(u)$.
\end{theorem}
 For $\delta>0$, the Young functions 
$$\bar{\Phi}(t)=\frac{t^{q'}}{\log(\mathsf{e}+t)^{1+\frac{q}{q'}\delta}} \ \ \text{and} \ \ \bar{\Psi}(t)=\frac{t^{p}}{\log(\mathsf{e}+t)^{1+\frac{p}{p'}\delta}}$$
satisfy $\bar{\Phi}\in B_{q'}$ and $\bar{\Psi}(t)\in B_p$ respectively.  The associate Young functions become 
$$\Phi(t)\eqsim t^q\log(\mathsf{e}+t)^{\frac{q}{q'}+\delta} \ \ \text{and} \ \ {\Psi}(t)\eqsim {t^{p'}}{\log(\mathsf{e}+t)^{\frac{p'}{p}+\delta}}$$
and thus the condition
\begin{equation}\label{doublelogbump} \sup_Q |Q|^{\frac{\al}{n}+\frac1q-\frac1p}\|u^\frac1q\|_{L^q(\log L)^{\frac{q}{q'}+\delta}(Q)}\|v^{-\frac1p}\|_{L^{p'}(\log L)^{\frac{p'}{p}+\delta}(Q)}<\infty\end{equation}
is sufficient for $I_\al:L^p(v)\ra L^q(u)$.  The work of Cruz-Uribe and the second author \cite{CM1,CM2} improved condition \eqref{Apqbump} twofold when $p<q$.  First, they were able to work with separated bump conditions for the strong type inequality.  Second they introduced a better integrability condition known as the $B_{p,q}$ condition:
$$\int_1^\infty\frac{\Phi(t)^{\frac{q}{p}}}{t^q}\frac{dt}{t}<\infty.$$
It is easy to see that $B_p\subsetneq B_{p,q}$ when $p<q$ and thus $B_{p,q}$ is a more general condition that allows for a larger class of Young functions. They proved the following theorem.
\begin{theorem}[\cite{CM2}] \label{sepbump} Suppose $1<p<q<\infty$ and $(u,v)$ are weights.  If $\Phi$ and $\Psi$ are Young functions that satisfy $\bar{\Phi}\in B_{q',p'}$, $\bar{\Psi}\in B_{p,q}$, and $(u,v)$ satisfy
\begin{multline*}\sup_Q|Q|^{\frac{\al}{n}+\frac1q-\frac1p}\|u^\frac1q\|_{L^\Phi(Q)}\left(\dashint_Q v^{1-p'}\right)^{\frac1{p'}}\\+\sup_Q|Q|^{\frac{\al}{n}+\frac1q-\frac1p}\left(\dashint_Q u\right)^{\frac1{q}}\|v^{-\frac1p}\|_{L^\Psi(Q)}<\infty\end{multline*}
then $I_\al:L^p(v)\ra L^q(u).$ 
\end{theorem}
As mentioned above, Theorem \ref{sepbump} improves Theorem \ref{Perbump} in two ways, but only holds in the case $p<q$.  For example, the following separated logarithmic bump condition 
\begin{multline*}\sup_Q|Q|^{\frac{\al}{n}+\frac1q-\frac1p}\|u^\frac1q\|_{L^q(\log L)^{\frac{q}{p'}+\delta}(Q)}\left(\dashint_Q v^{1-p'}\right)^{\frac1{p'}}\\+\sup_Q|Q|^{\frac{\al}{n}+\frac1q-\frac1p}\left(\dashint_Q u\right)^{\frac1{q}}\|v^{-\frac1p}\|_{L^{p'}(\log L)^{\frac{p'}{q}+\delta}(Q)}<\infty\end{multline*}
is sufficient for the boundedness of $I_\al:L^p(v)\ra L^q(u)$ when $p<q$. Not only are there two terms with simple $L^{p'}$ and $L^q$ averages, but there are also smaller powers on the logarithms than those in \eqref{doublelogbump}, since $\frac{q}{p'}<\frac{q}{q'}$ and $\frac{p'}{q}<\frac{p'}{p}$, if $p<q$.

When $p=q$ the $B_{p,q}$ condition collapses to the $B_p$ condition. Surprisingly, in the diagonal case $p=q$, the best sufficient separated bump conditions for the boundedness of $I_\al$ from $L^p(v)$ to $L^p(u)$ are 
\begin{multline*}\sup_Q|Q|^{\frac{\al}{n}}\|u^\frac1p\|_{L^p(\log L)^{2p-1+\delta}(Q)}\left(\dashint_Q v^{1-p'}\right)^{\frac1{p'}}\\+\sup_Q|Q|^{\frac{\al}{n}}\left(\dashint_Q u\right)^{\frac1{p}}\|v^{-\frac1p}\|_{L^{p'}(\log L)^{2p'-1+\delta}(Q)}<\infty.\end{multline*}
This result can be found in the book \cite{CMP}. The powers on the logarithms are not ideal and should hold for $\Phi(t)=t^p\log(\mathsf{e}+t)^{p-1+\delta}$ and $\Psi(t)=t^{p'}\log(\mathsf{e}+t)^{p'-1+\delta}$. It is an open question as to whether or not the larger power can be removed. We summarize the results as applications of our main theorem.

\begin{corollary} Suppose $T$ is any of the operators mentioned in the statement of Corollary \ref{maintwocor}.  If $1<p< q<\infty$ and $(u,v)$ are weights that satisfy the separated bump condition
\begin{multline*}\sup_Q|Q|^{\frac{1}{n}+\frac1q-\frac1p}\|u^\frac1q\|_{L^\Phi(Q)}\left(\dashint_Q v^{1-p'}\right)^{\frac1{p'}}\\+\sup_Q|Q|^{\frac{1}{n}+\frac1q-\frac1p}\left(\dashint_Q u\right)^{\frac1{q}}\|v^{-\frac1p}\|_{L^\Psi(Q)}<\infty,\end{multline*}
then $T:\dot W^{1,p}(v)\ra L^q(u)$ with
$$\|Tf\|_{L^q(u)}\leq C\|\nabla f\|_{L^p(v)}.$$
\end{corollary} 

\begin{corollary} Suppose $T$ is any of the operators mentioned in the statement of Corollary \ref{maintwocor} and $1<p<\infty$. If $(u,v)$ satisfy either the separated logarithmic bump condition
\begin{multline*}\sup_Q|Q|^{\frac1{n}}\|u^\frac1p\|_{L^p(\log L)^{2p-1+\delta}(Q)}\left(\dashint_Q v^{1-p'}\right)^{\frac1{p'}}\\+\sup_Q|Q|^{\frac{1}{n}}\left(\dashint_Q u\right)^{\frac1{p}}\|v^{-\frac1p}\|_{L^{p'}(\log L)^{2p'-1+\delta}(Q)}<\infty,\end{multline*}
 for some $\delta>0$, or the joint bump condition 
$$\sup_Q |Q|^{\frac1{n}}\|u^{\frac1p}\|_{L^\Phi(Q)}\|v^{-\frac1p}\|_{L^\Psi(Q)}<\infty$$
for some $\bar{\Phi}\in B_{p'}$ and $\bar{\Psi}\in B_p$, then $T:\dot W^{1,p}(v)\ra L^p(u)$ with
$$\|T f\|_{L^p(u)}\leq C\|\nabla f\|_{L^p(v)}.$$
\end{corollary}

\section{Endpoint results} \label{endpt}

As mentioned in the introduction, the weak-type endpoint behavior for several of the operators under consideration is very difficult and many open questions still remain. For example, even in the non-weighted setting, the inequalities
$$\|\mathcal M_{\Sn^{n-1}}f\|_{L^{n',\infty}(\R^n)}\leq C\|\nabla f\|_{L^1(\R^n)}$$
and
$$\| M_{L^{n',1}}f\|_{L^{n',\infty}(\R^n)}\leq C\|\nabla f\|_{L^1(\R^n)}$$
are new and do not follow from known mapping properties of the operators.  

When $1\leq p\leq q$ it turns out that one of the testing conditions characterizes the $I_\al:L^p(v)\ra L^{q,\infty}(u)$ boundedness \cite{Saw1}.  In fact there is a version of these testing conditions when $p=1<q$, but we will not pursue this here.  Instead we examine two other weak type inequalities for $I_1$. 

Let $\mu$ be a measure. Then $\|\cdot\|_{L^{n',\infty}(\mu)}$ is equivalent to a norm since $n'>1$, and by the Minkowski integral inequality we see that
$$\|I_1 g\|_{L^{n',\infty}(\mu)}\leq C\int_{\R^n}|g(y)| \big\| |\cdot -y|^{1-n}\big\|_{L^{n',\infty}(\mu)}\,dy\eqsim \int_{\R^n}|g|(M\mu)^{\frac1{n'}},$$
(see \cite{PR} for the specific calculations). If we define $w\in A_{1,n'}$ to mean $w^{n'}\in A_1$ with constant
$$[w]_{A_{1,n'}}=[w^{n'}]_{A_1}=\left\|\frac{M(w^{n'})}{w^{n'}}\right\|_{L^\infty(\R^n)},$$
then letting $d\mu=w^{n'}dx$ we obtain 
$$\|I_1 g\|_{L^{n',\infty}(w^{n'})}\leq C\int_{\R^n}|g|M(w^{n'})^{\frac{1}{n'}}\leq C[w]^{\frac1{n'}}_{A_{1,n'}}\int_{\R^n} |g| w.$$
and this bound is sharp (see \cite{LMPT}).  We have the following corollary.

\begin{corollary}  If $\mu$ is a meausure and $T$ is any operator as in the statement of Corollary \ref{maintwocor}, then
\begin{equation}\label{weakn} \|Tf\|_{L^{n',\infty}(\mu)}\leq C\int_{\R^n}|\nabla f|(M\mu)^{\frac1{n'}}\end{equation}
for any $f\in C^\infty_c(\R^n)$.  In particular, for $w\in A_{1,n'}$
$$ \|Tf\|_{L^{n',\infty}(w^{n'})}\leq C[w]^\frac{1}{n'}_{A_{1,n'}}\int_{\R^n}|\nabla f|w. $$
\end{corollary}

Inequality \eqref{weakn} is the weak-type estimate in $L^{q,\infty}$ for $q>1$, namely $q=n'=\frac{n}{n-1}$, and does not give any information about the $L^{1,\infty}$ behavior. Carro, the second author, Soria, and Soria \cite{CPSS} studied the weak-type (1,1) behavior of $I_\al$. They showed that the inequality $$\|I_\al g\|_{L^{1,\infty}(w)}\leq C\|g\|_{L^1(M_\al w)}$$
does \emph{not} hold in general and proved the following replacement using logarithmic maximal functions.  Given any Young function $\Phi$ let
$$M_{L^\Phi}f(x)=\sup_{Q\ni x}\|f\|_{L^\Phi(Q)}.$$
If $\Phi(t)=t\log(\mathsf{e}+t)^\ep$ we write $M_{L^\Phi}=M_{L(\log L)^\ep}$.

\begin{theorem}[\cite{CPSS}] Suppose $w$ is a weight and $0<\al<n$.  Then for any $\ep>0$
$$\|I_\al g\|_{L^{1,\infty}(w)}\leq C\int_{\R^n}|g| M_\al(M_{L(\log L)^\ep} w).$$
\end{theorem}

Hence we have the following corollary.

\begin{corollary} \label{endpointbdd} For any $\ep>0$ and any operator $T$ as in the statement of Corollary \ref{maintwocor}, we have
\begin{equation}\label{endpointbddineq}\|Tf\|_{L^{1,\infty}(w)}\leq C\int_{\R^n}|\nabla f| M_1(M_{L(\log L)^\ep} w)\end{equation}
for any $f\in C^\infty_c(\R^n)$.
\end{corollary}

\end{document}